\documentclass{amsart} 
\usepackage{newlattice}
\usepackage{graphics}
\usepackage{tikz} 
\tikzset{every picture/.style={line width=0.7pt}}
\newcommand{\mycircle}[1]
{\draw[fill=white,line width=1pt] (#1) circle[radius=1.0mm]}
\usepackage{fixltx2e} 
\usepackage{amssymb} 
\usepackage{latexsym} 
\usepackage{eucal}
\usepackage{enumerate}
\usepackage{xspace}

\DeclareMathOperator{\Conj}{Conj}
\DeclareMathOperator{\length}{length}

\newtheorem{theorem}{Theorem}
\newtheorem{lemma}[theorem]{Lemma}

\theoremstyle{definition}

\newtheorem{note}[theorem]{Note}

\newcommand{\swing}{\mathbin{\raisebox{2.0pt}{\rotatebox{160}{$\curvearrowleft$}}}}

\begin{document}
\title[Congruences and prime-perspectivities]
{Congruences and prime-perspectivities\\in finite lattices} 
\author{G. Gr\"{a}tzer} 
\address{Department of Mathematics\\
   University of Manitoba\\
   Winnipeg, MB R3T 2N2\\
   Canada}
\email[G. Gr\"atzer]{gratzer@me.com}
\urladdr[G. Gr\"atzer]{http://server.math.umanitoba.ca/homepages/gratzer/}

\date{\today}
\keywords{Prime-perspective, congruence, congruence-perspective, perspective, prime interval.}
\subjclass[2010]{Primary: 06B10}

\begin{abstract}
In a finite lattice, a congruence spreads 
from a prime interval to another 
by a sequence of congruence-perspectivities 
through \emph{intervals of~arbitrary size},
by a 1955 result of J. Jakub\'ik. 

In this note, I introduce 
the concept of \emph{prime-perspectivity} 
and prove the Prime-projectivity Lemma: 
a congruence spreads from a prime interval to another
by a sequence of prime-perspectivities through \emph{prime intervals}.

A planar semimodular lattice is \emph{slim} 
if it contains no $\mathsf{M}_3$ sublattice. 
I~introduce the Swing Lemma,
a very strong version of the Prime-projectivity Lemma
for slim, planar, semimodular lattices.
\end{abstract}

\maketitle

\section{Introduction}
To describe the congruence lattice, 
$\Con L$, of a finite lattice $L$, 
note that a~prime interval $\fp$ 
generates a join-irreducible congruence, 
$\con{\fp}$, and conversely; 
see, for~instance, the discussion on pages 213 and 214 of LTF 
(reference \cite{LTF}). 
So if we can determine when $\con{\fp} \geq \con{\fq}$
(that is, $\fq$ is collapsed by $\con{\fp}$) holds
for the prime intervals $\fp$ and $\fq$ of $L$, 
then we know the distributive lattice $\Con L$ up to isomorphism.
  
This is accomplished by the following result 
of J.~Jakub\'ik \cite{jJ55}
(see Lemma 238 in~LTF), 
where $\cproj$ is congruence-projectivity, 
see Section~\ref{S:Preliminaries}.
 
\begin{lemma}\label{L:cproj}%Lemma~\ref{L:cproj}
Let $L$ be a finite lattice and let $\fp$ and $\fq$ be prime intervals in $L$.
Then~$\fq$ is collapsed by $\con{\fp}$ if{f} $\fp \cproj \fq$.
\end{lemma}

Jakub\'ik's condition can be visualized 
using Figure~\ref{F:Prime-projectivity}; 
we may have to go through intervals of arbitrary size
to get from $\fp$ to $\fq$.

In this note, I introduce the concept of prime-perspectivity.
Let $L$ be a finite lattice 
and let $\fp$ and $\fq$ be prime intervals of $L$. 
The first diagram in Figure~\ref{F:intro}
depicts \emph{$\fp$ down-perspective to $\fq$}, 
in formula, $\fp \perspdn \fq$. 
We define \emph{$\fp$ up-perspective to $\fq$}, 
in formula, $\fp \perspup \fq$, dually. 
In the second diagram in Figure~\ref{F:intro},
$\fq$ is collapsed by $\con{\fp}$,
but we cannot get from $\fp$ to $\fq$ 
by down- and up-perspectivities. 

We introduce a more general step: 
$\fp$~is \emph{prime-perspective down} to $\fq$ if
$\fp$ is down-perspective to~$[0_\fp \mm 1_\fq, 1_\fq]$
and $\fq$ is contained in $[0_\fp \mm 1_\fq, 1_\fq]$.
In other words, we have an $\SN 5$ sublattice:
$\set{0_\fp \mm 1_\fq, 0_\fp, 0_\fq, 1_\fq, 0_\fp \jj 0_\fq}$,
so $\fp$ is down-perspective to~$[0_\fp \mm 1_\fq, 1_\fq]$
and $\fq$ is contained in $[0_\fp \mm 1_\fq, 1_\fq]$.

\begin{figure}[thb]%Figure~\ref{F:intro} 
\centerline{\includegraphics{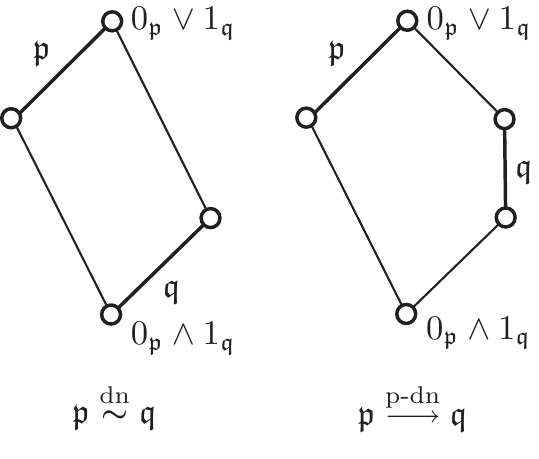}}
\caption{Introducing prime-perspectivity}\label{F:intro}
\end{figure}

We formalize this as follows: 
the binary relation $\fp$ \emph{prime-perspective down} to $\fq$, 
in formula, $\fp \pperspdn \fq$, 
is defined as $1_\fq \leq 1_\fp$,
$\fp \perspdn [0_\fq, 1_\fp \jj 0_\fq]$,
and $0_\fp \mm 1_\fq \leq 0_\fq$. 

So if $\fp \pperspdn \fq$, then $\fp$ and $\fq$ generate
an $\SN 5$, as in the second diagram of Figure~\ref{F:intro},
or a $\SC 2^2$, as in the first diagram of Figure~\ref{F:intro},
or $\SC 2$, if $\fp = \fq$.

We define \emph{prime-perspective up}, $\fp \pperspup \fq$, dually.
Let \emph{prime-perspective}, $\fp \ppersp \fq$, 
mean that $\fp \pperspup \fq$ or $\fp \pperspdn \fq$ and let
\emph{prime-projective}, $\fp \pproj \fq$, 
be the transitive extension of $\ppersp$.

\begin{note} 
A prime-perspectivity $\fp \pperspdn \fq$ that is not a
perspectivity is \emph{established}
by an $\SN 5$ sublattice.
\end{note}

\begin{figure}[thb]%Figure~\ref{F:Prime-projectivity} 
\centerline{\includegraphics{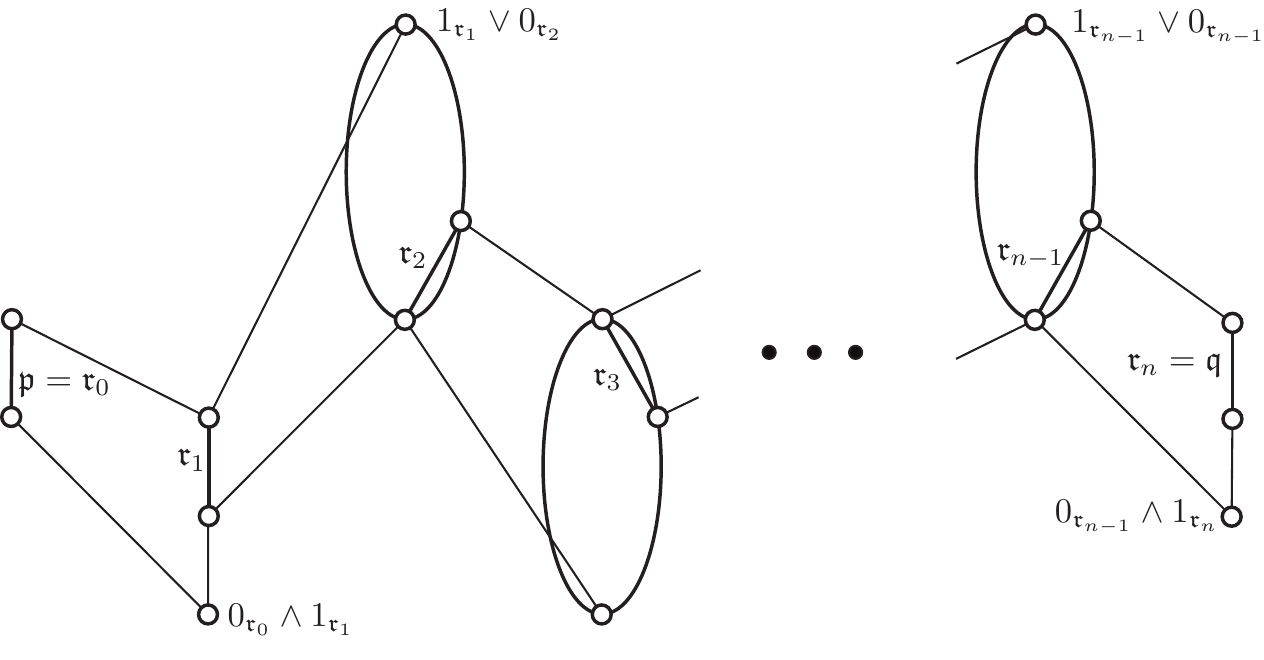}}
\caption{Prime-projectivity: $\fp \pproj \fq$}
\label{F:Prime-projectivity}
\end{figure}

\begin{note} 
The relation $\fp \ppersp \fq$ holds if{}f 
the elements $\set{0_\fp, 0_\fq, 1_\fp, 1_\fq}$ 
satisfy a simple (universal) condition: 
$\fp$ and $\fq$ are perspective or they generate
an $\SN 5$ sublattice as in Figure~\ref{F:intro}.
\end{note}

\begin{note}\label{N:bold}
In the figures, bold lines designate coverings.
\end{note}

Now we state our result: we only have to go through prime intervals 
to spread a congruence from a prime interval to another.
 
\begin{lemma}[Prime-Projectivity Lemma]\label{L:prime}
%Lemma~\ref{LT:prime}
Let $L$ be a finite lattice and 
let $\fp$ and $\fq$ be distinct prime intervals in $L$.  
Then $\fq$ is collapsed by $\con{\fp}$ 
if{}f $\fp \pproj \fq$, that is, 
if{}f there exists a sequence of pairwise distinct prime intervals
$\fp = \fr_0, \fr_1, \dots, \fr_n = \fq$ satisfying
\begin{equation}\label{E:ppthm}%\eqref{E:ppthm}
\fp = \fr_0 \ppersp \fr_1 \ppersp \dotsm \ppersp \fr_n = \fq.
\end{equation}
\end{lemma}

Let us call a lattice $L$ an \emph{SPS lattice}, 
if it is slim, planar, semimodular; see
G.~Gr\"atzer and E.~Knapp~\cite{GKn08a}.
We state a strong form of the
Prime-projectivity Lemma for SPS lattices,
using the concept of a swing.

For the prime intervals $\fp, \fq$ of an SPS lattice $L$, 
we define a binary relation:
$\fp$~\emph{swings} to $\fq$, in formula, $\fp \swing \fq$,
if $1_\fp = 1_\fq$, 
this element covers at least three elements,
and $0_\fq$ is not the left-most or right-most element
covered by $1_\fp = 1_\fq$. 
We call the element $1_\fp = 1_\fq$
the \emph{hinge} of the swing.

See Figure~\ref{F:n7+} for two examples; in the first, the hinge covers three elements, in~the second, five elements. 

\begin{figure}[hbt]%Figure~\ref{F:n7+}
\centerline{\includegraphics{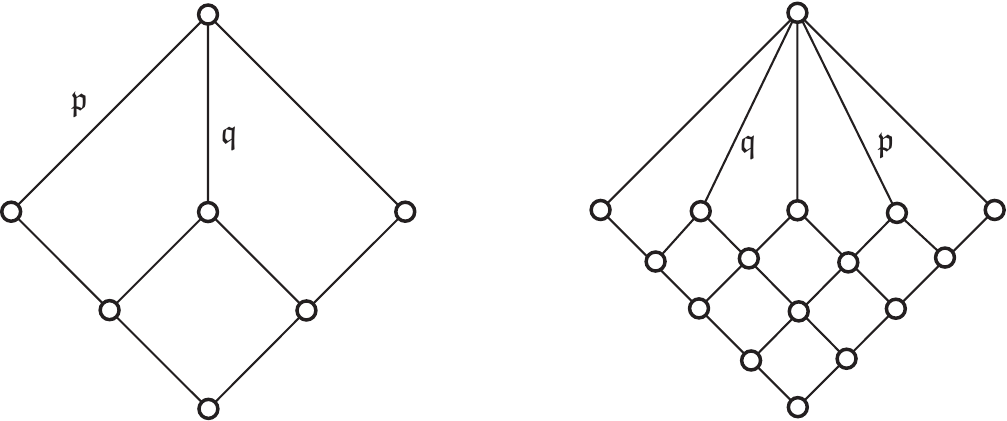}} 
\caption{Swings, $\fp \protect\swing \fq$}\label{F:n7+}
\end{figure}

\begin{lemma}[Swing Lemma]\label{L:SPSproj}
%Lemma~\ref{L:SPSproj}
Let $L$ be an SPS lattice 
and let $\fp$ and $\fq$ be distinct prime intervals in $L$. 
Then $\fq$ is collapsed by $\con{\fp}$ if{}f 
there exists a prime interval~$\fr$ 
and sequence of pairwise distinct prime intervals
\begin{equation}\label{Eq:sequence}%\eqref{Eq:sequence}
\fr = \fr_0, \fr_1, \dots, \fr_n = \fq
\end{equation}
such that $\fp$ is up perspective to $\fr$, and 
$\fr_i$ is down perspective to or swings to $\fr_{i+1}$
for $i = 0, \dots, n-1$. 
In addition, the sequence \eqref{Eq:sequence} also satisfies 
\begin{equation}\label{E:geq}%\eqref{E:geq}
   1_{\fr_0} \geq 1_{\fr_1} \geq \dots \geq 1_{\fr_n}.
\end{equation}
\end{lemma}

\begin{figure}[thb]%Figure~\ref{F:traj} 
\centerline{\includegraphics{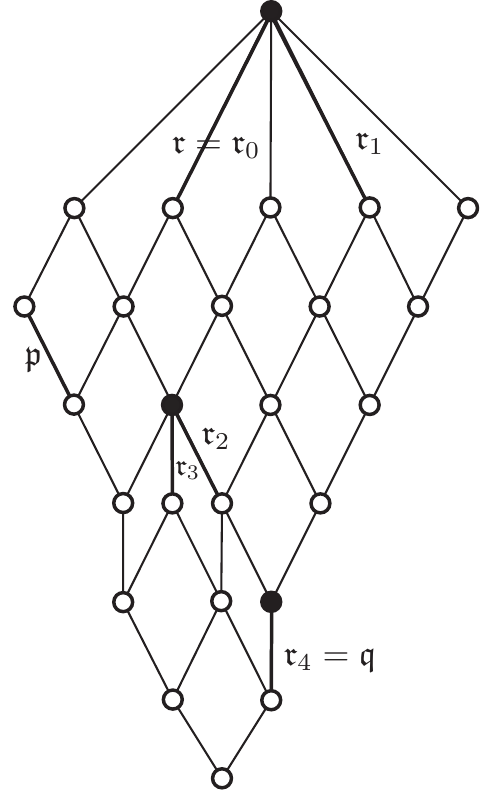}}
\caption{Illustrating the Swing Lemma}\label{F:traj}
\end{figure}

The Swing Lemma is easy to visualize. 
Perspectivity up is ``climbing'', 
perspectivity down is ``sliding''. 
So we get from $\fp$ to $\fq$ by climbing once
and then alternating sliding and swinging.
In Figure~\ref{F:traj}, 
we climb up from $\fp$ to $\fr = \fr_0$, 
swing from $\fr_0$ to $\fr_1$,
slide down from $\fr_1$ to $\fr_2$,
swing from $\fr_2$ to $\fr_3$,
and finally slide down from $\fr_3$ to $\fr_4$.
 
Section~\ref{S:Preliminaries} recalls some basic concepts and notation.
In~Section~\ref{S:Prime}, I prove the Prime-projectivity Lemma. 
In~Section~\ref{S:application},
the Prime-projectivity Lemma is applied to verify 
a result of~mine with E. Knapp \cite{GKn08a}.
Section~\ref{S:Swing} deals with the Swing Lemma.
 
\section{Preliminaries}\label{S:Preliminaries}
%Section~\ref{S:Preliminaries}
We use the concepts and notation of LTF.
For an ideal $I$, we use the notation $I = [0_I, 1_I]$.
%In a planar lattice, a \emph{$4$-cell} is a covering $\SC 2^2$ sublattice 
%with no elements inside. 

We recall that $[a,b] \persp [c,d]$ denotes \emph{perspectivity}, 
$[a,b] \perspup [c,d]$ and 
$[a,b] \perspdn [c,d]$ perspectivity up and down, 
see Figure~\ref{F:cong1}; 
$[a,b] \proj [c,d]$ denotes \emph{projectivity}, the transitive closure of perspectivity.

$[a,b] \cpersp [c,d]$ denotes \emph{congruence-perspectivity},
$[a,b] \cperspup [c,d]$ and $[a,b] \cperspdn [c,d]$ 
denote congruence-perspectivity up and down, see Figure~\ref{F:cong2}; 
$[a, b] \cproj [c, d]$ denotes \emph{congruence-projectivity}, 
the transitive closure of congruence-perspectivity.

\begin{figure}[tbh]
  \centerline{\includegraphics{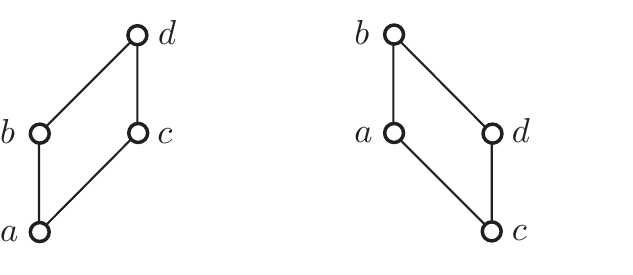}}
  \caption{$[a,b] \persp [c, d]$: $[a,b] \perspup [c, d]$ and $[a,b] \perspdn [c, d]$}\label{F:cong1}

\bigskip

\bigskip

\centerline{\includegraphics{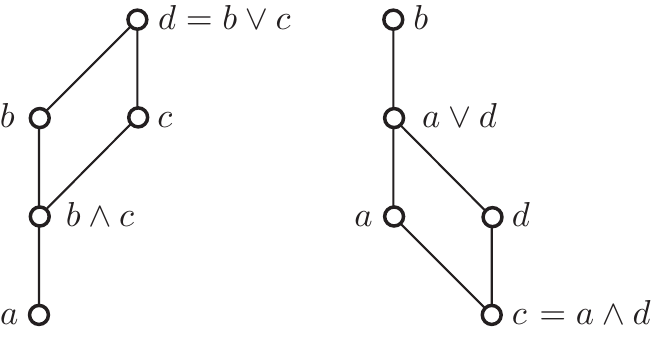}}
\caption{$[a,b] \cperspup [c, d]$ and $[a,b] \cperspdn [c, d]$}\label{F:cong2}
\end{figure}%Figure~\ref{F:cong2}

\section{Proving the Prime-projectivity Lemma}\label{S:Prime}
%Section~\ref{S:Prime}

To prove the Prime-projectivity Lemma, 
let $\fp$ and $\fq$ be prime intervals in a finite lattice $L$ 
with $\con{\fp} \geq \con{\fq}$.
By Lemma~\ref{L:cproj}, there is a sequence of congruence-perspectivities 
\begin{equation}\label{E:pr-perp}
\fp = I_0 \cpersp I_1 \cpersp \dotsm \cpersp I_m = \fq.
\end{equation}

To get from \eqref{E:pr-perp} to \eqref{E:ppthm}, by induction on $m$, 
it is sufficient to prove the following statement.

\begin{lemma}\label{L:cover0}%Lemma~\ref{L:cover0}
Let $L$ be a finite lattice and let $I \cpersp J$ be intervals of $L$. 
Let $\fb$ be a prime interval in $J$. 
Then there exists a prime interval $\fa \ci I$ 
satisfying $\fa \pproj \fb$.
\end{lemma}

\begin{proof}
By duality, we can assume that $I \cperspup J$.
We prove the statement by induction on $\length(I)$, 
the length of $I$ (that is, the length of the longest chain in $I$).
Note that if $I'$ is an interval properly contained in $I$, then
$\length(I') <\length (I)$.

For the induction base, let $I$ be prime. Then take $\fa=I$. 

For the induction step, we can assume that $I$ is not prime 
and that the statement is proved for intervals shorter than $I$.

Without loss of generality, we can assume that 
\begin{enumeratei}
\item $J \nci I$---otherwise, take $\fa = \fb \ci I$;

\item $0_{\fb} = 0_J$---otherwise, 
replace $J$ with $[0_{\fb}, 1_J]$;

\item $0_I=1_I\mm 0_{\fb}$---otherwise, replace $I$ with $[1_I \mm 0_J, 1_I] \ci I$.
\end{enumeratei}

Since $I$ is not a prime interval, there is an element $u \in I$ 
satisfying  $0_I \prec u < 1_I$. If $1_{\fb}\leq u \jj 0_{\fb}$, 
then take $\fa=[0_I,u]$; clearly, $\fa \pperspup \fb$.

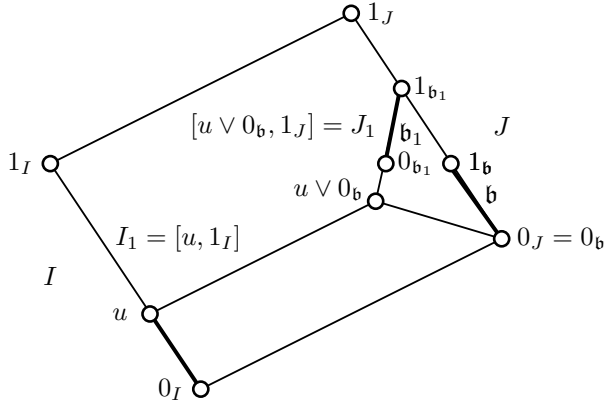
\begin{figure}[hbt] 
\centerline{\begin{tikzpicture}
%\draw[help lines] (0,0) grid (6,6);
\draw (2,0)--(0,3)--(4,5)--(6,2)--(2,0);
\draw (1.325,1)--(4.325,2.5);
\draw (6,2)--(4.325,2.5)--(4.66,4);
\draw[line width = 1.5pt](2,0)--(1.33,1);
\draw[line width = 1.5pt](4.45,3)--(4.65,4);
\draw[line width = 1.5pt](6,2)--(5.285,3);
\mycircle{2,0};
\node at (2-.4,0) {$0_I$}; 
\mycircle{6,2};
\node at (6+.8,2) {$0_J = 0_\fb$};
\mycircle{0,3};
\node at (0-.4,3) {$1_I$};
\mycircle{4,5};
\node at (4+.4,5) {$1_J$};
\mycircle{1.325,1};
\node at (1.325-.4,1) {$u$};
\mycircle{5.325,3};
\node at (5.325+.4,3) {$1_\fb$};
\mycircle{4.325,2.5};
\node at (5.325+.4,3) {$1_\fb$};
\mycircle{4.46,3};
\node at (4.46+.4,3) {$0_{\fb_1}$};
\mycircle{4.67,4};
\node at (4.67+.4,4) {$1_{\fb_1}$};
\node at (6,3.5) {$J$};
\node at (0,1.5) {$I$};
\node at (1.7,2) {$I_1 = [u, 1_I]$};
\node at (3.1,3.5) {$[u \jj 0_\fb, 1_J] = J_1$};
\node at (5.86,2.6) {$\fb$};
\node at (4.8,3.4) {$\fb_1$};
\node at (3.7,2.66) {$u \jj 0_\fb$};
\end{tikzpicture}}
\caption{Proving the Prime-projectivity Lemma}
\label{F:Proving}
%Figure~\ref{F:Proving} 
\end{figure}

Therefore, we can additionally assume that 
\begin{enumeratei}
\item[(iv)] $u \jj 0_{\fb} \parallel 1_{\fb}$,
\end{enumeratei}
see Figure~\ref{F:Proving}. Then 
\[
   u \jj 1_\fb = (u \jj 0_\fb) \jj 1_\fb > u \jj 0_\fb.
\]
So we can define a prime interval~$\fb_1$ 
with 
\[
   u \jj 0_{\fb} \leq 0_{\fb_1}\! \prec 1_{\fb_1} = u \jj 1_{\fb}.
\] 
Observe that 
\begin{equation}\label{E:onestep}
\fb_1 \perspdn \fb. 
\end{equation}
Applying~the induction hypotheses to 
$I_1 = [u ,1_I]$, $J_1 = [0_{\fb_1}, 1_J]$, and $\fb_1$, 
we obtain a prime interval $\fa \ci I_1 \ci I$ 
satisfying $\fa \pproj \fb_1$. 
Combining $\fa \pproj \fb_1$ and \eqref{E:onestep}, 
we obtain that $\fa \pproj \fb$,
as required.
\end{proof}

The proof actually verified more than stated. 
Every step of the induction adds no more than one prime interval 
to the sequence.
 
\section{An application}\label{S:application}%Section~\ref{S:application}
In \cite{GKn08a}, I proved with  E. Knapp the following result:

\begin{theorem}\label{T:n5}%Theorem~\ref{T:n5}
Let $L$ be a finite lattice and 
let $\fp$ and $\fq$ be prime intervals in $L$ such that  
$\con{\fp} \succ  \con{\fq}$ in the order $\Conj{L}$ of join-irreducible congruences of $L$.
Then there exist a sublattice $\SN 5$ of $L$---depicted by 
Figure~\ref{F:Prime-perspectivity}---and prime intervals $\fp_0$ and $\fq_0$ in~$\SN 5$ satisfying
$\con{\fp} = \con{\fp_0}$ and 
$\con{\fq} = \con{\fq_0}$.
\end{theorem}

\begin{proof}
Let $\fp$ and $\fq$ be prime intervals in~$L$ 
such that  $\con{\fp} \succ \con{\fq}$ in the order
$\Conj$ of join-irreducible congruences of $L$. 
By~the Prime-projectivity Lemma,
there exists prime intervals $\fr_i$, 
for $0 \leq i \leq m$, such that 
\eqref{E:ppthm} holds. Clearly, 
\[
   \con{\fp}= \con{\fr_0} \geq \con{\fr_1} \geq \cdots \geq  
       \con{\fr_m} = \con{\fq}.
\]
So there is a (unique) $0 \leq k < m$ such that 
\begin{align*}
      \con{\fr_k} &= \con{\fp},\\
  \con{\fr_{k+1}} &= \con{\fq}.
\end{align*} 
Therefore, $\fr_k=\fp_0$ and $\fr_{k+1}=\fq_0$ 
 satisfy the statement of Theorem~\ref{T:n5}. 
\end{proof}
\section{The Swing Lemma}\label{S:Swing}%Section~\ref{S:Swing}
\subsection{Some basic concepts}\label{S:basic}%Section~\ref{S:basic}
For SPS lattices, we stated the Swing Lemma, 
a~strong form of the Prime-projectivity Lemma. 
To discuss its proof, we need a few more concepts and results.
(For a more detailed overview,
see G. Cz\'edli and G. Gr\"atzer~\cite{CGa}
and G. Gr\"atzer~\cite{gG13b}.)

An element $c$ of an SPS lattice $L$ is a \emph{left corner},
if it is a doubly-irreducible element 
on the left boundary of $L$
(excluding $0_L$ and $1_L$), 
its unique cover $c^*$ covers exactly two elements, 
and dually, for the lower cover $c_*$.
We define a \emph{right corner} symmetrically.

An SPS lattice $L$ is called \emph{rectangular} 
if it has a unique left corner $c_l$ and a~unique right corner
$c_r$, and these two elements are complementary.
See G. Gr\"atzer and E.~Knapp~\cite{GKn07}--\cite{GKn10}
for these concepts and the two observations in the next paragraph.

Every SPS lattice $L$ has a congruence-preserving extension
to a rectangular lattice $R$, which we obtain by adding corners.
By the same token, every SPS lattice~$L$ can be obtained 
from a rectangular lattice $R$ by removing corners.

\subsection{Verifying the Swing Lemma}\label{S:Verifying}%Section~\ref{S:Verifying}
We have two proofs of the Swing Lemma.

G. Cz\'edli applies in \cite{gC14a} 
the Trajectory Coloring Theorem for Slim Rectangular Lattices 
of G. Cz\'edli~\cite[Theorem 7.3]{gC14} 
to prove the Swing Lemma for rectangular lattices.

In view of the discussion at the end of 
Section~\ref{S:basic}, to complete the proof of the Swing Lemma,
it~is sufficient to prove the following result. 

\begin{lemma}\label{L:omit}%Lemma~\ref{L:omit}
Let $L$ be an SPS lattice. 
Let $c$ be a corner of $L$ and let $L' = L - \set{c}$,
a~sublattice of $L$.
Let us assume that the Swing Lemma holds for the lattice $L$. 
Then the Swing Lemma holds also for the lattice $L'$.
\end{lemma}

\begin{proof}
Let $\fp$ and $\fq$ be prime intervals in~$L'$
such that $\fq$ is collapsed by  $\con{\fp}$ in~$L'$.
Then we also have that $\fq$ is collapsed by $\con{\fp}$ in~$L$.
Since the Swing Lemma is assumed to hold in $L$,
therefore, there exists a sequence 
of pairwise distinct prime intervals
\eqref{Eq:sequence} in $L$. 
We can assume that we choose the sequence in
\eqref{Eq:sequence} so that 
$n$ and $\length [1_p, 1_{r_0}]$ are minimal.
To~show that the Swing Lemma holds in~$L'$, 
we prove that every prime interval 
$\fr_i$ in  the sequence \eqref{Eq:sequence} is in $L'$. 

Let $\set{c, c_*, c^*, d}$ 
be the (unique) covering square in $L$ containing $c$.
The lattice~$L$ has exactly two prime intervals 
not in $L'$, namely,
$\ft_* = [c_*, c]$ and $\ft^* = [c, c^*]$. 
We claim that neither appears in \eqref{Eq:sequence};
the verification of this claim completes the proof of the lemma

Let us assume that $\ft^*$ appears in the sequence \eqref{Eq:sequence} for $\fp$ and $\fq$. 
Since $c^*$ covers exactly two elements 
by the definition of a corner, 
it is not the hinge of a swing, 
so the step to $\ft^*$ must be a perspectivity up 
and the step from $\ft^*$ must be a perspectivity down;
therefore, the sequence can be shortened, contradicting the minimality of $n$.

Let $\ft_*$ appear in the sequence.
Since $c$ is doubly irreducible, 
it is not the hinge of a swing, 
so the step to $\ft_*$ must be a perspectivity down.
For the same reason, the next step, back into $L'$,  
must be a perspectivity up; therefore, 
the sequence can be shortened, 
contradicting again the minimality of $n$.
\end{proof}

The second proof in G. Gr\"atzer~\cite{gG14e}
is more ``elementary''; it proves the Swing Lemma directly
for SPS lattices.

\end{document}